\documentclass[11pt,oneside]{article}

\pdfoutput=1

\usepackage{latexsym}
\usepackage{shortvrb}
\usepackage{graphicx}
\usepackage{fancyhdr}
\usepackage{amsfonts}
\usepackage{amsmath}
\usepackage{amssymb}
\usepackage{mathrsfs}
\usepackage{makeidx}
\usepackage[all]{xy}
\usepackage{tikz}
\usepackage{amsthm}

\usepackage[T1]{fontenc}
\usepackage{amsfonts}
\usepackage{amsmath}

\newtheorem{theorem}{Theorem}[section]
\newtheorem{lemma}[theorem]{Lemma}
\newtheorem{proposition}[theorem]{Proposition}

\newtheorem{corollary}[theorem]{Corollary}

\theoremstyle{definition}

\newcommand{\blackged}{\hfill$\blacksquare$}
\newcommand{\whiteged}{\hfill$\square$}
\newcounter{proofcount}

\renewenvironment{proof}[1][\proofname.]{\par
  \ifnum \theproofcount>0 \pushQED{\whiteged} \else \pushQED{\blackged} \fi%
  \refstepcounter{proofcount}
  \normalfont 
  \trivlist
  \item[\hskip\labelsep
        \itshape
    {\bf\em #1}]\ignorespaces
}{%
  \addtocounter{proofcount}{-1}
  \popQED\endtrivlist
}

\begin{document}

\begin{center}
\textbf{\large{On Direct Sum Decompositions of Krull-Schmidt Artinian Modules}}
\end{center}

\begin{center}
Juan Orendain
\end{center}

\noindent We study direct sum decompositions of modules satisfying the descending chain condition on direct summands. We call modules satisfying this condition Krull-Schmidt artinian. We prove that all direct sum decompositions of Krull-Schmidt artinian modules refine into finite indecomposable direct sum decompositions and we prove that this condition is strictly stronger than the condition of a module admitting finite indecomposable direct sum decompositions. We also study the problem of existence and uniqueness of direct sum decompositions of Krull-Schmidt artinian modules in terms of given classes of modules. We present also brief studies of direct sum decompositions of modules with deviation on direct summands and of modules with finite Krull-Schmidt length.

\tableofcontents

\section{Introduction}
	
The study of rings and modules satisfying one or two of the chain conditions on distinguished sets of ideals or submodules is a fundamental part of the theory of rings and their modules. We study modules satisfying the descending (equivalently ascending) chain condition on direct summands. We say that a module satisfying these conditions is Krull-Schmidt artinian. We thus study direct sum decompositions of Krull-Schmidt artinian modules and direct sum decompositions of modules satisfying related finiteness conditions. Notice that 
in this setting the ascending chain condition on direct summands already appears in [3]. We now sketch the content of this paper. In section 2 we present examples of Krull-Schmidt artinian rings and modules, and we compare the condition of a ring or a module being Krull-Schmidt artinian against other finiteness conditions. In section 3 we study direct sum decompositions of Krull-Schmidt artinian modules. We prove that all direct sum decompositions of Krull-Schmidt artinian modules refine into finite indecomposable direct sum decompositions, and using a result of Facchini and Herbera, we prove that this condition is strictly stronger than the condition of a module only having indecomposable direct sum decompositions. We also study the problem of existence and uniqueness of direct sum decompositions of Krull-Schmidt artinian modules in terms of given classes of modules. We obtain direct sum decompositions of Krull-Schmidt artinian modules in terms of well known classes of modules, and in some cases we prove that these direct sum decompositions are canonical. In sections 4 and 5 we study modules with deviation on direct summands and modules with finite Krull-Schmidt length respectively, we characterize modules satisfying these conditions, and we briefly study their direct sum decompositions. Examples of non-Krull-Schmidt artinian modules with deviation on direct summands and of Krull-Schmidt artinian modules with infinite Krull-Schmidt length are provided. Finally, in section 6 we consider what we call the $AKS_i$ conditions. We compare these conditions under the presence of finiteness conditions studied in previous sections.
We will assume throughout that a unital ring with unit $R$ is chosen. The word \textit{module} will thus mean a left $R$-module, and the word \textit{morphism} will mean a morphism in the category of left $R$-modules, $R$-Mod. We will write $\mathcal{S}(M)$ for the set of all direct summands of a module $M$. Observe that $\mathcal{S}(M)$ thus defined is a poset with order relation given by inclusion.   

\section{Krull-Schmidt Artinian Modules}

In this first section we present examples of Krull-Schmidt artinian rings and modules and we make a few remarks on closure properties of the class of all Krull-Schmidt artinian modules.

Modules with finite Goldie dimension, modules with finite dual Goldie dimension, and thus modules with Krull (equivalently dual Krull) dimension are easily seen to be examples of Krull-Schmidt artinian modules. Thus also artinian and noetherian modules are examples of Krull-Schmidt artinian modules. A semisimple module is Krull-Schmidt artinian if and only if it is artinian. More generally, a module with the exchange property is Krull-Schmidt artinian if and only if it admits a finite indecomposable direct sum decomposition. Any torsion free indecomposable abelian group of infinite rank is an example of non-finitely generated and non-finitely cogenerated Krull-Schmidt artinian module with infinite Goldie dimension. 
Thus the condition of a module being Krull-Schmidt artinian does not imply, in general, neither the condition of a module being finitely generated, nor the condition of a module being finitely cogenerated, nor the condition of a module having finite Goldie dimension. It follows also that submodules of Krull-Schmidt artinian modules need not, in general, be Krull-Schmidt artinian. Quotients of Krull-Schmidt artinian modules also need not, in general, be Krull-Schmidt artinian, which is easily seen from the fact that $\mathbb{Q}$ is uniform but $\mathbb{Q}/\mathbb{Z}$ is not Krull-Schmidt artinian. It is easily seen that direct summands of Krull-Schmidt artinian modules are Krull-Schmidt artinian, we prove in section 3 however, that finite direct sums of Krull-Schmidt artinian modules need not, in general, be Krull-Schmidt artinian. The ring $\mathbb{Z}_2^{\mathbb{N}}$, as a module over itself, is cyclic but not Krull-Schmidt artinian, thus the condition of a module being finitely generated does not, in general, imply the condition of a module being Krull-Schmidt artinian.

Given a ring $E$ and two idempotents $e,e'\in E$, we will write $e\leq e'$ when $ee'=e=e'e$. Observe that the set of all idempotents of $E$, equipped with this relation, is a poset. We will say that the ring $E$ is a Krull-Schmidt artinian ring if the poset of idempotents of $E$ satisfies the descending chain condition. The following proposition puts the condition of a module $M$ being Krull-Schmidt artinian, completely in terms of the ring of endomorphisms, $End(M)$, of $M$.

\begin{proposition}
Let $M$ be a module. Let $E$ denote the ring of endomorphisms, $End(M)$, of $M$. The following conditions are equivalent:
\begin{enumerate}
\item $M$ is Krull-Schmidt artinian.
\item The left $E$-module $_EE$ is Krull-Schmidt artinian.
\item The right $E$-module $E_E$ is Krull-Schmidt artinian.
\item $E$ is a Krull-Schmidt artinian ring.
\item The poset of all idempotents of $E$ satisfies the ascending chain condition.
\item $E$ has no infinite subsets of orthogonal idempotents.
\end{enumerate}
\end{proposition}

\begin{proof}
Let $e,e'\in E$ be two idempotents. Then $e\leq e'$ if and only if for any ring $S$ and any $S$-module $N$ such that $E=End_S(N)$, the direct summand $e(N)$ of $N$ is a direct summand of the direct summand $e'(N)$ of $N$. The result follows from this and from the fact that for any ring $S$ and any $S$-module $N$ such that $E=End_S(N)$, every direct summand of $N$ is the image $e(N)$ of an idempotent $e\in E$.
\end{proof}

\noindent It follows, from proposition 2.1 that rings with finite left (right) Goldie dimension and semilocal rings (i.e. rings with finite dual Goldie dimension) are Krull-Schmidt artinian. In particular, left (right) Goldie rings, left (right) artinian rings, left (right) noetherian rings, and more generally, rings with left (right) Krull dimension are all Krull-Schmidt artinian. An exchange ring is Krull-Schmidt artinian if and only if it is semiperfect and all integral domains are Krull-Schmidt artinian. Moreover, any module $M$ such that the ring of endomorphisms, $End(M)$, of $M$ satisfies any of the conditions above is itself Krull-Schmidt artinian.

\section{Direct Sum Decompositions}
In this section we study direct sum decompositions of Krull-Schmidt artinian modules. It is easily seen that Krull-Schmidt artinian modules admit finite indecomposable direct sum decompositions. In the next proposition we prove that Krull-Schmidt artinian modules satisfy a stronger condition. Recall first that given two direct sum decompositions of a module $M$, $M=\bigoplus_{\alpha \in A}M_\alpha$, and $M=\bigoplus_{\beta \in B}N_\beta$, it is said that the direct sum decomposition $M=\bigoplus_{\alpha \in A}M_\alpha$ refines into $M=\bigoplus_{\beta \in B}N_\beta$ if for each $\alpha\in A$ there exists a direct sum decomposition $M_\alpha=\bigoplus_{\gamma_\alpha \in C_\alpha}M_{\gamma_\alpha}$ of $M_\alpha$ such that the two direct sum decompositions $M=\bigoplus_{\alpha \in A}(\bigoplus_{\gamma_\alpha\in C_\alpha}M_{\gamma_\alpha})$ and $M=\bigoplus_{\beta \in B}N_\beta$ of $M$ are equal. In this case, it is also said that $M=\bigoplus_{\beta \in B}N_\beta$ is a refinement of $M=\bigoplus_{\alpha \in A}M_\alpha$.

\begin{proposition}
If $M$ is Krull-Schmidt artinian, then all direct sum decompositions of $M$ refine into finite indecomposable direct sum decompositions.
\end{proposition}

\begin{proof}
Let $M=\bigoplus_{\alpha\in A}M_\alpha$ be a direct sum decomposition of $M$. Clearly the direct sum decompostion $M=\bigoplus_{\alpha\in A}M_\alpha$ refines into a finite indecomposable direct sum decomposition of $M$ if and only if $A$ is finite and $M_\alpha$ admits finite indecomposable direct sum decompositions for every $\alpha\in A$. $A$ is finite because the set $\left\{M_\alpha:\alpha\in A\right\}$ is an independent set of direct summands of $M$, and since for every $\alpha\in A$, $M_\alpha$, being direct summand of the Krull-Schmidt artinian module $M$, is itself Krull-Schmidt artinian, it admits finite indecomposable direct sum decompositions. This concludes the proof.
\end{proof}

\noindent To see that the condition of all direct sum decompositions of a module refining into finite indecomposable direct sum decompositions is indeed stronger than the condition of the module only admitting finite indecomposable direct sum decompositions we use the following result due to A. Facchini and D. Herbera. A proof of this result can be found in [6, corollary 5.6].

\begin{corollary}
For every additive submonoid, $A$, of $\mathbb{N}$ there exists a ring $R$ and an $R$-module $M$ such that the module $M^n$ admits finite indecomposable direct sum decompositions if and only if $n\in A$. 
\end{corollary}

\noindent Using corollary 3.2 we construct an example of a ring $R$ and an $R$-module $N$ such that $N$ admits finite indecomposable direct sum decompositions, but such that not all its direct sum decompositions refine into finite idecomposable direct sum decompositions. To this end, write $A=2\mathbb{N}$. By corollary 3.2 there exists then a ring $R$ and an $R$-module $M$ such that while the module $M^2$ admits finite indecomposable direct sum decompositions, $M$ does not. Thus $N=M^2$ is a module admitting finite indecomposable direct sum decompositions, but such that its direct sum decomposition $N=M\oplus M$ does not refine into a finite indecomposable direct sum decomposition. Note that from this example and from proposition 3.2 it follows that while both the class of all modules such that all their direct sum decompositions refine into finite indecomposable direct sum decompositions, and the class of all Krull-Schmidt artinian modules, are closed under direct summands, niether one of them is closed under taking finite direct sums.

We next investigate the problem of decomposing modules in terms of given classes of modules. More precisely, given a module $M$ and a class of modules $\Omega$, we investigate the problem of existence and uniqueness of direct sum decompositions of $M$, of the form $M=A\oplus B$
such that $A\in \Omega$ and such that $B$ has no non-trivial direct summands in $\Omega$. The next proposition establishes, under certain assumptions on the class $\Omega$, existence in the case $M$ is Krull-Schmidt artinian.

\begin{proposition}
Let $\Omega$ be a class of modules closed under finite direct sums. If $M$ is Krull-Schmidt artinian then there exists a direct sum decomposition $M=A\oplus B$ of $M$ such that $A\in \Omega$ and such that $B$ has no non-trivial direct summands in $\Omega$.  
\end{proposition}

\begin{proof}
Suppose $\Omega\cap\mathcal{S}(M)\neq\left\{0\right\}$. From the fact that $M$ is Krull-Schmidt artinian it follows that $\Omega\cap\mathcal{S}(M)\setminus\left\{0\right\}$ admits non-trivial maximal elements. Let $A$ be maximal in $\Omega\cap\mathcal{S}(M)\setminus \left\{0\right\}$. Now, let $B$ be such that $M=A\oplus B$. Suppose $B$ has non-trivial direct summands in $\Omega$. Let $L$ be a non-trivial direct summand of $B$ in $\Omega$. Then  $A<A\oplus L$, and $A\oplus L\in \Omega\cap\mathcal{S}(M)\setminus\left\{0\right\}$, which contradicts the maximality of $A$ in $\Omega\cap\mathcal{S}(M)\setminus\left\{0\right\}$. We conclude that $B$ has no non-trivial direct summands in $\Omega$.
\end{proof}

\noindent Given a class of modules $\Omega$, we can always consider its closure, $\tilde{\Omega}$, under the operation of taking finite direct sums, that is, $\tilde{\Omega}$ will be the class of all modules admitting finite direct sum decompositions into modules in $\Omega$. $\tilde{\Omega}$ is thus closed under finite direct sums, and $\tilde{\Omega}=\Omega$ if and only if $\Omega$ is itself closed under finite direct sums. Thus, the condition of $\Omega$ being closed under finite direct sums on proposition 3.3 can be eliminated at the expense of obtaining direct sum decompositions of the form: $M=A\oplus B$ where $B$ does not admit non-trivial direct summands in $\Omega$ and $A\in\tilde{\Omega}$, that is $A$, might not be in $\Omega$, but it does admit finite direct sum decompositions into modules in $\Omega$. Both the statement of proposition 3.3 and the statement obtained by the above observation are clearly equivalent.

Given two classes of modules $\Omega_1$ and $\Omega_2$, closed under finite direct sums, and a Krull-Schmidt artinian module $M$, it is easily seen, from the proof of proposition 3.3, that if $\mathcal{S}(M)\cap\Omega_1\subseteq \mathcal{S}(M)\cap\Omega_2$, then, for every direct sum decomposition $M=A_1\oplus B_1$ of $M$ such that $A_1\in\Omega_1$ and such that $B_1$ has no non-trivial direct summands in $\Omega_1$, there exists a direct sum decomposition $M=A_2\oplus B_2$ of $M$ such that $A_2\in\Omega_2$, such that $B_2$ has no non-trivial direct summands in $\Omega_2$, and such that $A_1$ is a direct summand of $A_2$.

Given $\Psi$, a logical proposition taking values on the class of all $R$-modules, we write $\left[\Psi(M)\right]$ for the truth value of $\Psi$ on the module $M$, that is, $\left[\Psi(M)\right]$ will be equal to $1$ or $0$ depending on whether $\Psi(M)$ is true or false. We use the convention that $0 \cdot\infty=0$. As a consequence of proposition 3.3 we have the following corollary.

\begin{corollary}
Let $\Theta$ be a class of modules. Let $\xi$ be a function with domain $\Theta$ and codomain $\mathbb{N}\cup\left\{\infty\right\}$. Suppose $\Theta$ is closed under finite direct sums. If the equation
\[\xi(A\oplus B)=max\left\{\xi(A),\xi(B)\right\}\]
holds for all $A,B\in\Theta$, then, for all $M\in\Theta$, there exists a finite direct sum decomposition 
\[M=(\bigoplus_{i=1}^nA_i)\oplus B\]
of $M$, such that $\xi(A_i)=\left[A_i\neq\left\{0\right\}\right]i$ for all $1\leq i\leq n $ and $\xi(B)=\left[B\neq \left\{0\right\}\right]\infty$. 
\end{corollary}

\begin{proof}
Let $\Omega$ be the class of all $N\in\Theta$ such that there exists a finite direct sum decomposition $N=\bigoplus_{i=1}^nA_i$
of $N$ such that $\xi(A_i)=\left[A_i\neq\left\{0\right\}\right]i$ for all $1\leq i\leq n $. We prove $\Omega$ is closed under finite direct sums. Let $N_j\in\Omega$, $j=1,2$. Let $N_j=\bigoplus_{i=1}^{n_j}A_{i,j}$ be a direct sum decomposition of $N_j$ such that $\xi(A_{i,j})=\left[A_{i,j}\neq\left\{0\right\}\right]i$ for all $1\leq i\leq n_j $ and $j=1,2$. Suppose, without any loss of generality, that $n_1\leq n_2$. Then, since 
\[\xi(A_{i,1}\oplus A_{i,2})=max\left\{\xi(A_{i,1},A_{i,2})\right\}=[A_{i,1}\oplus A_{i,2}\neq\left\{0\right\}]i\]

for all $i\leq n_1$, if we denote by $B_i$, $A_{i,1}\oplus A_{i,2}$ if $1\leq n_1$ and $A_{i,2}$ if $n_1+1\leq i\leq n_2$, then, the direct sum decomposition 
\[N_1\oplus N_2=\bigoplus_{i=1}^{n_2}B_i\]
of $N_1\oplus N_2$ is such that $\xi(B_i)=[B_i\neq\left\{0\right\}]i$ for all $i$. This proves our claim. The result follows from this and proposition 3.3.  
 
\end{proof}

\noindent Given a module $M$, we write, $pdimM$, $EdimM$, $KdimM$, and $dKdimM$ for the projective, injective, Krull, and dual Krull dimension of $M$ respectively. Observe that the functions $pdim$ and $Edim$ that associate to each module $M$, $pdimM$ and $EdimM$ respectively, and the functions $Kdim$ and $dKdim$ that associate to each module $M$ with Krull dimension, $KdimM$ or $dKdimM$ respectevely, when this ordinal is finite, or $\infty$ when it is not, satisfy the condition of corollary 3.4 with $\Theta=R$-Mod in the first two cases, and with $\Theta$ the class of all modules with Krull dimension in the remaining cases. From this and from corollary 3.4 follows the following corollary.

\begin{corollary}
For any module $M$
\begin{enumerate}
\item If $M$ is Krull-Schmidt artinian, then there exists a finite direct sum decomposition $M=(\bigoplus_{i=1}^nA_i)\oplus B$ of $M$ such that $pdimA_i=\left[A_i\neq\left\{0\right\}\right]i$ for all $1\leq i\leq n$ and $pdimB=\left[B\neq \left\{0\right\}\right]\infty$.
\item If $M$ is Krull-Schmidt artinian, then there exists a finite direct sum decomposition $M=(\bigoplus_{i=1}^nA_i)\oplus B$ of $M$ such that $EdimA_i=\left[A_i\neq\left\{0\right\}\right]i$ for all $1\leq i\leq n$ and $EdimB=\left[B\neq \left\{0\right\}\right]\infty$.
\item If $M$ has Krull dimension then there exists a finite direct sum decomposition $M=(\bigoplus_{i=1}^nA_i)\oplus B$ of $M$ such that $KdimA_i=\left[A_i\neq\left\{0\right\}\right]i$ for all $1\leq i\leq n$ and $KdimB=\left[B\neq\left\{0\right\}\right]\infty$.
\item If $M$ has Krull dimension then there exists a finite direct sum decomposition $M=(\bigoplus_{i=1}^nA_i)\oplus B$ of $M$ such that $dKdimA_i=\left[A_i\neq\left\{0\right\}\right]i$ for all $1\leq i\leq$ and $dKdimB=\left[B\neq\left\{0\right\}\right]\infty$.
\end{enumerate}
\end{corollary}

\noindent Two direct sum decompositions $M=\bigoplus_{\alpha\in A}M_\alpha$, and $M=\bigoplus_{\beta\in B}N_\beta$ of a module $M$, are said to be decomposition-isomorphic if there exists a bijection $\phi :A\to B$ such that for each $\alpha\in A$, $M_\alpha$ and $N_{\phi (\alpha)}$ are isomorphic. In the next proposition we investigate the question of uniqueness, up to decomposition-isomorphisms, of direct sum decompositions as in proposition 3.3. Recall before that a module $M$ is said to be strongly indecomposable if its endomorphism ring, $End(M)$, is local. A module $M$ is thus strongly indecomposable if and only if it is indecomposable and has the finite exchange property.

\begin{theorem}
Let $\Omega$ be a class of modules such that:
\begin{enumerate}
\item $\Omega$ is closed under finite direct sums.
\item $\Omega$ is closed under direct summands.
\item Every indecomposable $N\in\Omega$ is strongly indecomposable.
\end{enumerate} 
Then, if $M$ is Krull-Schmidt artinian there exists a direct sum decomposition $M=A\oplus B$ of $M$, unique up to decomposition-isomorphisms such that $A\in \Omega$ and such that $B$ has no non-trivial direct summands in $\Omega$. 
\end{theorem}

\begin{proof}
The existence of a direct sum decomposition $M=A\oplus B$ of $M$ such that $A\in \Omega$ and such that $B$ has no non-trivial direct summands in $\Omega$ has already being established in proposition 3.3. We prove now that $M$ admits only one such direct sum decomposition up to decomposition-isomorphisms. Let $M=A'\oplus B'$ be another direct sum decomposition of $M$ such that $A'\in\Omega$ and such that $B'$ has no non-trivial direct summands in $\Omega$. Since $A$ and $A'$ are both Krull-Schmidt artinian, then both admit finite indecomposable direct sum decompositions. Let $A=\bigoplus_{i=1}^nA_i$ and $A'=\bigoplus_{j=1}^mA'_j$ be finite indecomposable direct sum decompositions of $A$ and $A'$ respectively. Since $A_i\in\Omega$, and $A'_j\in \Omega$ for all $1\leq i\leq n$ and $1\leq j\leq m$, and they are all indecomposable, then the $A_i$'s and the $A'_j$'s are all strongly indecomposable. Suppose $m\geq n$. By the fact that $A_1$ has the finite exchange property, and by the fact that that $B'$ has no non-trivial direct summands in $\Omega$ it follows that there exists $1\leq j_1\leq m$ such that

\[M=A_1\oplus (\bigoplus_{j\neq j_1}A'_j)\oplus B'\]

Now, again, using the fact that each $A_i$ has the finite exchange property and doing induction on $n$ it follows that there exists $I\subseteq \left\{1,...,m\right\}$ of cardinality $n$ such that

\[M=(\bigoplus_{i=1}^nA_i)\oplus (\bigoplus_{j\in\left\{1,...m\right\}\setminus I}A'_j)\oplus B'\]

Now, since $B$ has no non-trivial direct summands in $\Omega$, using now the fact that $A'_j$ also has the exchange property for every $j\in \left\{1,...,m\right\}\setminus I$, we see that $I=\left\{1,...,m\right\}$. Thus, $n=m$ and 

\[(\bigoplus_{i=1}^nA_i)\oplus B'=M=(\bigoplus_{j=1}^nA'_j)\oplus B'\]

We conclude from this, from the fact that direct sum decompositions on the left side of both sides of the equation are strongly indecomposable, and from the fact that $B$ and $B'$ are both isomorphic to $M/\bigoplus _{i=1}^nA_i$ that $M=A\oplus B$ and $M=A'\oplus B'$ are both decomposition-isomorphic.
\end{proof}

\begin{corollary}
Suppose $M$ is Krull-Schmidt artinian. Then there exists a direct sum decomposition $M=E\oplus F$ of $M$, unique up to decomposition-isomorphisms, such that $E$ has the finite exchange property and such that no non-trivial direct summand of $F$ has the finite exchange property. Moreover, for any $\Omega$ and $M=A\oplus B$ as in theorem 3.1, $A$ is direct summand of $E$.  
\end{corollary}

\begin{proof}
The result follows from theorem 3.6 and the remarks made after proposition 3.3.
\end{proof}

\noindent The class of all semisimple modules, the class of all modules with finite length, the class of all Fitting modules, the class of all injective modules, the class of all pure injective modules, and the class of all $\Sigma$-pure injective modules all satisfy the conditions of theorem 3.6. Moreover, if the ring $R$ is semiperfect, then the class of all projective modules also satisfies the conditions of theorem 3.6, and if moreover, $R$ is local, then the class of all free modules satisfies the conditions of theorem 3.6. The following corollary follows.

\begin{corollary}
Suppose $M$ is Krull-Schmidt artinian. Then:
\begin{enumerate}
\item There exists a direct sum decomposition $M=A\oplus B$ of $M$, unique up to decomposition-isomorphisms, such that $A$ is semisimple and such that $B$ has no non-trivial simple direct summands. 
\item There exists a direct sum decomposition $M=A\oplus B$ of $M$, unique up to decomposition-isomorphisms, such that $A$ has finite length and such that all non-trivial direct summands of $B$ have infinite length.
\item There exists a direct sum-decomposition $M=A\oplus B$ of $M$, unique up to decomposition-isomorphisms, such that $A$ is Fitting and such that no non-trivial direct summand of $B$ is Fitting.
\item There exists a direct sum decomposition $M=A\oplus B$ of $M$, unique up to decomposition-isomorphisms, such that $A$ is injective and such that $B$ has no non-trivial injective submodules.
\item There exists a direct sum decomposition $M=A\oplus B$ of $M$, unique up to decomposition-isomorphisms, such that $A$ is pure injective and such that $B$ has no non-trivial pure injective direct summands.
\item There exists a direct sum decomposition $M=A\oplus B$ of $M$, unique up to decomposition-isomorphisms, such that $A$ is $\Sigma$-pure injective and such that $B$ has no non-trivial $\Sigma$-pure injective direct summands.
\item If the ring $R$ is semiperfect then there exists a direct sum decomposition $M=A\oplus B$ of $M$, unique up to decomposition-isomorphisms such that $A$ is projective and such that $B$ has no non-trivial projective direct summands.
\item If the ring $R$ is local then there exists a direct sum decomposition $M=A\oplus B$ of $M$, unique up to decomposition-isomorphisms such that $A$ is free and such that $B$ has no non-trivial projective direct summands.
\end{enumerate}
\end{corollary}

\begin{corollary}
For any module $M$ we have:
\begin{enumerate}
\item If $M$ is artinian, then there exists a direct sum decomposition $M=A\oplus B$ of $M$, unique up to decomposition-isomorphisms, such that $A$ is noetherian and such that $B$ has no non-trivial noetherian direct summands.
\item If $M$ is noetherian, then there exists a direct sum decomposition $M=A\oplus B$ of $M$, unique up to decomposition-isomorphisms, such that $A$ is artinian and such that $B$ has no non-trivial artinian direct summands.
\end{enumerate}
\end{corollary}

\begin{proof}
The result follows from corollary 3.8, and from the fact that both artinian and noetherian modules are Krull-Schmidt artinian, together with the fact that modules with finite length are precisely those modules which are artinian and noetherian.
\end{proof}

\section{Modules with Deviation on Direct Summands}

In this section we study modules with deviation on direct summands. We characterize modules satisfying this condition and we present a brief study of their direct sum decompositions.

Recall that the deviation of a poset $A$, $devA$, is defined as follows: $devA$ is equal to $-1$ if $A$ is a discrete poset (i.e. if the relation $a\leq b$ in $A$ implies that $a=b$), and $devA$ is equal to an ordinal $\alpha$, if $A$ does not have deviation $\beta$, for any $\beta<\alpha$, and if for every descending chain $a_1\geq a_2\geq ...$ in $A$, the interval $[a_n,a_{n+1}]$, as a subposet of $A$, has deviation $<\alpha$ for all but finitely many $n$. Thus a poset $A$ satisfies the descending chain condition if and only if $devA\leq 0$. A module $M$ is thus Krull-Shmidt artinian if and only if $\mathcal{S}(M)$ has deviation $\leq 0$. Note that in this case the dual poset of $\mathcal{S}(M)$, $\mathcal{S}(M)^*$, also has deviation $\leq 0$. The following proposition generalizes this.

\begin{proposition}
Let $M$ be a module. Suppose $\mathcal{S}(M)$ has deviation. Then
\[dev\mathcal{S}(M)=dev\mathcal{S}(M)^*\]
\end{proposition}

\begin{proof}
Suppose $\mathcal{S}(M)$ has deviation $\alpha$. We prove, by transfinite induction on $\alpha$, that $dev\mathcal{S}(M)^*\leq\alpha$. The case in which $\alpha=-1$ is trivial. Suppose now that this is true for all $\beta<\alpha$. Let $M_1\leq M_2\leq...$ be an ascending chain in $\mathcal{S}(M)$ (i.e. a decending chain in $\mathcal{S}(M)^*$). We construct, recursively, a descending chain $N_1\geq N_2\geq ...$ in $\mathcal{S}(M)$ such that $M=M_n\oplus N_n$ for all $n$. Let $N_1$ be such that $M=M_1\oplus N_1$. Suppose now that for some $n$, we have constructed a chain $N_1\geq ...N_n$ in $\mathcal{S}(M)$ such that $M=M_i\oplus N_i$ for all $1\leq i\leq n$. It is clear, by the modular law, that $M_{n+1}=M_n\oplus(M_{n+1}\cap N_n)$. Thus $M_{n+1}\cap N_n$ is a direct summand of $M$. It follows that $M_{n+1}\cap N_n$ is a direct summand of $N_n$. Let $N_{n+1}$ be such that 
\[N_n=(M_{n+1}\cap N_n)\oplus N_{n+1}\]
It is clear that $N_{n+1}$ thus defined, satisfies the equation $M=M_{n+1}\oplus N_{n+1}$.

By the way the chain $N_1\geq N_2\geq...$ was constructed, it is immediate that all three modules 
\[M_{n+1}/M_n, M_{n+1}\cap N_n,N_n/N_{n+1}\]
are  isomorphic. It follows that $\mathcal{S}(M_{n+1}/M_n)$ and $\mathcal{S}(N_n/N_{n+1})$ are isomorphic posets for all $n\in \mathbb{N}$. Now, since $dev\mathcal{S}(M)=\alpha$, then $dev\mathcal{S}(N_n/N_{n+1})<\alpha$ for all but finitely many $n$, from which it follows that $dev\mathcal{S}(M_{n+1}/M_n)<\alpha$ for all but finitely many $n$. Thus, from the induction hypothesis, we have that $dev\mathcal{S}(M_{n+1}/M_n)^*<\alpha$ for all but finitely $n$. From this, from the fact that for any two direct summands $L$ and $N$ of $M$ such that $L\leq N$, the interval $[L,N]$ is isomorphic to the poset $\mathcal{S}(N/L)$, and from the fact that the interval $[N,L]$ in $\mathcal{S}(M)^*$ is isomorphic to the poset $\mathcal{S}(N/L)^*$, it follows that $dev\mathcal{S}(M)^*\leq \alpha$.
One proves analogously that $dev\mathcal{S}(M)\leq dev\mathcal{S}(M)^*$. This completes the proof. 
\end{proof}

\noindent Arguing as in proposition 2.1 we have the following proposition.

\begin{proposition}
Let $M$ be a module. Let $E$ denote the ring of endomorphisms, $End(M)$, of $M$. The following conditions are equivalent:
\begin{enumerate}
\item $\mathcal{S}(M)$ has deviation.
\item $\mathcal{S}(_EE)$ has deviation.
\item $\mathcal{S}(E_E)$ has deviation.
\item The poset of all idempotents of $E$ has deviation.
\item The dual of the poset of all idempotents of $E$ has deviation.
\end{enumerate}
Moreover, if $M$ satisfies the conditions above, then 
\[dev\mathcal{S}(M)=dev\mathcal{S}(_EE)=dev{S}(E_E)\]
And this ordinal is equal to the deviation of the poset of idempotents of $E$ and to the deviation of the dual of the poset of idempotents of $E$.
\end{proposition}

\noindent Thus the condition of a module $M$ admitting deviation on direct summands can be put completely in terms of the ring of endomorphisms, $End(M)$, of $M$. The following is an example of a non-Krull-Schmidt artinian ring with deviation on direct summands: Let $D$ be an integral domain. Let $E$ be the subring $D+D^{(\mathbb{N})}$ of $D^{\mathbb{N}}$, that is, $E$ is the ring of all infinite sequences of elements of $D$, which are constant almost everywhere. The set of all idempotents of $E$ is equal to the set of all sequences of $0$'s and $1$'s which are constant everywhere. It is easily seen that this set has deviation equal to 1. Thus $E$ is a non-Krull-Schmidt artinian ring with deviation on direct summands.

We end this section with a brief study of direct sum decompositions of modules admitting deviation on direct summands. Observe, from the proof of proposition 3.1, that all direct sum decompositions of a module $M$ refine into finite indecomposable direct sum decompositions of $M$ if and only if all direct sum decompositions of $M$ are finite and all direct summands of $M$ admit finite indecomposable direct sum decompositions. The next result says that modules admiting deviation on direct summands satisfy a condition very close to this.

\begin{proposition}
Let $M$ be a module. Suppose $\mathcal{S}(M)$ has deviation. Then all direct sum decompositions of $M$ are finite and every direct summand of $M$ has a non-trivial indecomposable direct summand.
\end{proposition}

\begin{proof}
Suppose that $M$ admits infinite direct sum decompositions. Let $M=\bigoplus_{\alpha\in A}M_\alpha$ be an infinite direct sum decomposition of $M$. Suppose, without any loss of generality that $A=\mathbb{Q}$. Then the function 
\[x\mapsto \bigoplus_{\alpha\in\mathbb{Q}:\alpha\leq x}M_\alpha\]
defines a strictly increasing function from $\mathbb{R}$ to $\mathcal{S}(M)$. Thus, in this case $\mathcal{S}(M)$ has subposets order-isomorphic to $\mathbb{R}$. It follows, from the fact that the class of all posets with deviation is closed under subposets, and from the fact that $\mathbb{R}$ does not admit deviation [5, ex. 7.4] that $\mathcal{S}(M)$ does not have deviation. This proves the first statement of the proposition.

Suppose now that $\mathcal{S}(M)$ has deviation $\alpha$. We do induction on $\alpha$ to prove that every non-trivial direct summand of $M$ has non-trivial indecomposable direct summands. The case $\alpha\leq 0$ follows from proposition 3.1. Suppose now that the claim is true for all $\beta<\alpha$. Let $M_1$ be a non-trivial direct summand of $M$. Suppose there exists a striclty descending chain $M_1>M_2>...$ in $\mathcal{S}(M)$. Let $n$ be such that $dev\mathcal{S}(M_n/M_{n+1})<\alpha$. Since $M_n/M_{n+1}$ is a non-trivial direct summand of $M$, then, by the induction hypothesis, it follows that $M_n/M_{n+1}$ has non-trivial indecomposable direct summands. From this and from the fact that $M_n/M_{n+1}$ is direct summand of $M_1$ it follows that $M_1$ admits non-trivial indecomposable direct sumands. This concludes the proof.  
\end{proof}

\noindent Finally, observe that if a module $M$ is not Krull-Schmidt artinian, then there exists a chain $M_1<M_2<...$ in $\mathcal{S}(M)$ such that, for each $n$, $M_n$ admits direct sum decompositions of cardinality $n$. The following proposition says that, if we further assume that the poset $\mathcal{S}(M)$ admits deviation, then we may take $M_n$ to admit indecomposable direct sum decompositions of cardinality $n$ for each $n$. Note that this implies in turn that every non-Krull-Schmidt artinian module admitting deviation on direct sumands admits infinite independent sets of direct summands each of which is indecomposable.  

\begin{proposition}
Let $M$ be a module. Suppose $\mathcal{S}(M)$ has deviation. If $M$ is not Krull-Schmidt artinian, then there exists a chain $M_1<M_2<...$ in $\mathcal{S}(M)$ such that, for each $n\geq 1$, $M_n$ admits indecomposable direct sum decompositions of cardinality $n$.
\end{proposition}

\begin{proof}
Suppose $\mathcal{S}(M)$ has deviation $\alpha$. We do induction on $\alpha$. Suppose $M$ is not Krull-Schmidt artinian. Suppose then first that $\alpha=1$. Then, by proposition 4.1, the dual poset of $\mathcal{S}(M)$, $\mathcal{S}(M)^*$, also has deviation 1. Let $N_1<N_2<...$ be a strictly ascending chain in $\mathcal{S}(M)$. Let $k$ be such that $N_{k+i+1}/N_{k+i}$ is Krull-Schmidt artinian for all $i\geq 1$. Then $N_{k+i+1}/N_{k+i}$ admits indecomposable direct sum decompositions for all $i\geq 1$. Thus, if for each $i$, $L_i$ is an indecomposable direct summand of $N_{k+i+1}/N_{k+i}$ and for each $n$ we make $M_n=\bigoplus_{i=1}^nL_i$, then the chain $M_1<M_2<...$ satisfies the conditions of the proposition. This establishes the base of the induction.

Suppose now that $\alpha>1$ and that the result has be proven for all $\beta<\alpha$. Suppose $N_1>N_2>...$ is a strictly descending chain in $\mathcal{S}(M)$ such that there exists $k$ such that 
\[0<dev\mathcal{S}(N_k/N_{k+1})<\alpha\]
Then, $N_k/N_{k+1}$ is a direct summand of $M$, such that, by the induction hypothesis, there exists a chain $M_1<M_2<...$ in $\mathcal{S}(N)$ such that for each $n$, $M_n$ admits direct sum decompositions of cardinality $n$. This concludes the proof. 
\end{proof}

\section{Modules with Finite Krull-Schmidt Length}

\noindent We define the Krull-Schmidt length of a module $M$, $KS\ell(M)$, as the supremum of all cardinalities of direct sum decompositions of $M$. In this section we characterize modules with finite Krull-Schmidt length and we present a brief study on their direct sum decompositions.

Modules with finite Krull-Schmidt length are easily seen to be Krull-Schmidt artinian. Moreover, all examples of Krull-Schmidt artinian modules presented in section 2 also have finite Krull-Schmidt length. [7] provides an example of a ring $E$ such that the left $E$-module $_EE$ admits indecomposable direct sum decompositions of cardinality $n$ for all $n\geq 1$, but such that $E$ does not admit infinite sets of orthogonal idempotents. The module $_EE$ thus defined is an example of a Krull-Schmidt artinian module with infinite Krull-Schmidt length. We will regard modules with finite Krull-Schmidt length as those Krull-Schmidt artinian modules for which a rudimentary measure of complexity on direct sum decompositions, their Krull-Schmidt length, can be associated.

Observe that the class of all modules with finite Krull-Schmidt length is closed under direct summands. Observe also that by the arguments presented after corollary 3.2, this class is not closed under finite direct sums.

\begin{theorem}
Let $M$ be a module. The following conditions are equivalent:
\begin{enumerate}
\item $M$ has finite Krull-Schmidt length.
\item There exists $n\geq 1$ such that every independent subset of $\mathcal{S}(M)$ has cardinality at most $n$.
\item There exists $n\geq 1$ such that every strictly ascending chain of elements of $\mathcal{S}(M)$ has length at most $n$. 
\item There exists $n\geq 1$ such that every strictly descending chain of elements of $\mathcal{S}(M)$ has length at most $n$.
\item There exist strictly increasing functions from $\mathcal{S}(M)$ to $\mathbb{N}$.
\end{enumerate}
Moreover, if $M$ satisfies the conditions above, then $KS\ell(M)$ is the maximum of cardinalities of independent subsets of $\mathcal{S}(M)$, the maximum of lengths of strictly ascending chains in $\mathcal{S}(M)$, the maximum of lengths of strictly descending chains in $\mathcal{S}(M)$, and if we denote by $\Theta$ the set of striclty increasing functions from $\mathcal{S}(M)$ to $M$, ordered in terms of their images (i.e. $\alpha\leq\beta$ if $\alpha(N)\leq\beta(N)$ for all $N\in \mathcal{S}(M)$, with $\alpha,\beta\in\Theta$) then $\Theta$ has a minimum and it is equal to the function $KS\ell$ that associates to each $N\in\mathcal{S}(M)$ its Krull-Schmidt length $KS\ell(N)$.
\end{theorem}

\begin{proof}
The equivalences $1\Leftrightarrow 2$ and $3\Leftrightarrow4$ are trivial.

$1\Leftrightarrow4$: Suppose first that $M$ has finite  Krull-Schmidt length. Let $M_1>...>M_m$ be a strictly descending chain in $\mathcal{S}(M)$. Suppose, without any loss of generality, that $M_1=M$. If, for every $1\leq i\leq m-1$, $M'_i$ is such that $M_i =M'_i \oplus M_{i+1}$, and we let $M'_m$ be equal to $M_m$, then $M=\bigoplus_{i=1}^mM'_i$ is a direct sum decomposition of $M$, of cardinality $m$. It follows that $m\leq KS\ell(M)$. Suppose now that every strictly descending chain in $\mathcal{S}(M)$ has length at most $n$. We prove that $KS\ell(M)\leq n $. Let $M=\bigoplus_{i=1}^mM_i$ be a direct sum decomposition of $M$ such that $M_i\neq\left\{0\right\}$ for all $i$. If we let $M'_j=\bigoplus_{i=1}^{n-j}M_i$ for all $1\leq j\leq m$, then $M'_1>...>M'_m$ is a strictly descending chain in $\mathcal{S}(M)$, of length $m$. It follows that $m\leq n$. We conclude that $KS\ell(M)\leq n$.

$1\Leftrightarrow5$: Suppose $M$ has finite Krull-Schmidt length. Then the function $KS\ell :\mathcal{S}(M)\rightarrow \mathbb{N}$ that associates to each $N\in\mathcal{S}(M)$ its Krull-Schmidt length $KS\ell(N)$, is well defined. The function $KS\ell$ thus defined is strictly incresing. Suppose now that there exist strictly increasing functions from $\mathcal{S}(M)$ to $\mathbb{N}$. Let $\eta:\mathcal{S}(M)\rightarrow\mathbb{N}$ be a strictly increasing function. We make induction on $\eta(M)$ to prove that every direct sum decompositions of $M$ has cardinality at most $\eta(M)$. The case in which $\eta(M)=1$ is trivial. Suppose now that the result is true for all $k$ such that $1\leq k\leq \eta(M)-1$. Let $M=\bigoplus_{i=1}^mM_i$ be a direct sum decomposition of $M$. Clearly $\eta(\bigoplus_{i=1}^{m-1}M_i)+1\leq\eta(M)$, i.e. $\eta(\bigoplus_{i=1}^{m-1}M_i)\leq \eta(M)-1$. It follows, from the induction hypothesis, that all direct sum decompositions of $\bigoplus_{i=1}^{m-1}M_i$ have cardinality at most $\eta(M)-1$, in particular, $m-1\leq \eta(M)-1$, i.e. $m\leq \eta(M)$. We conclude that $M$ has finite Krull-Schmidt length.
The final statement of the theorem follows from $1\Leftrightarrow3$, $2\Rightarrow4$, and $4\Rightarrow 1$.
\end{proof}

\noindent It follows, from theorem 5.1 that if a module $M$ is such that there exists an additive function $\xi$ on $\mathcal{S}(M)$ (i.e. $\xi$ is such that $\xi(A\oplus B)=\xi(A)+\xi(B)$ for all $A,B\in\mathcal{S}(M)$ such that $A\oplus B\in\mathcal{S}(M)$), with image contained in $\mathbb{N}$, such that $\xi(N)\neq 0$ for all $N\in \mathcal{S}(M)$ such that $N\neq\left\{0\right\}$, then $M$ has finite Krull-Schmidt length. Moreover, the inequality
\[KS\ell(N)\leq \xi(N)\]
holds for all $N\in\mathcal{S}(M)$. In particular, if $M$ has finite length, finite Goldie dimension, or finite dual Goldie dimension, then the inequality
\[KS\ell(N)\leq min\left\{\ell(N),GdimN,dGdimN\right\}\]
holds for all $N\in\mathcal{S}(M)$. 

We define the Krull-Schmidt length of a ring $E$ ($KS\ell(E)$) as the supremum of the set of cardinalities of complete sets of orthogonal idempotents of $E$. Thus a ring $E$ has finite Krull-Schmidt length if and only if the set of cardinalities of sets of orthogonal idempotents is finite. Again, arguing as in proposition 2.1 we have the following proposition.

\begin{proposition}
Let $M$ be a module. Let $E$ denote the ring of endomorphisms, $End(M)$, of $M$. The following conditions are equivalent:
\begin{enumerate}
\item $M$ has finite Krull-Schmidt length.
\item The left $E$-module $_EE$ has finite Krull-Schmidt length.
\item The right $E$-module $E_E$ has finite Krull-Schmidt length.
\item The ring $E$ has finite Krull-Schmidt length.
\item There exists $n\geq 1$ such that all strictly ascending chains of idempotents of $E$ have length at most $n$.
\item There exists $n\geq 1$ such that all strictly descending chains of idempotents of $E$ have length at most $n$.
\item There exists $n\geq 1$ such that every set of orthogonal idempotents of $E$ has cardinality at most $n$.
\end{enumerate}
Moreover, if $M$ satisfies the conditions above, then 
\[KS\ell(M)=KS\ell(_EE)=KS\ell(E_E)=KS\ell(E)\]
and this number is the maximum of lengths of striclty ascending chains of idempotents of $E$, the maximum of lengths of strictly descending chains of idempotents of $E$, and the maximum of cardinalities of sets of orthogonal idempotents.
\end{proposition}

\noindent It follows, from proposition 5.2 that if a ring $E$ is left (right) artinian, has finite left (right) Goldie dimension, or is semilocal (i.e. if $E$ has finite dual Goldie dimension), then the inequality
\[KS\ell(E)\leq min\left\{\ell(E),Gdim(E),dGdim(E)\right\}\]
holds, where $\ell(E)$, $Gdim(E)$ and $dGdim(E)$ denote the length, the Goldie dimension, and the dual Goldie dimension of $E$ as a right or left $E$-module. Moreover, if $M$ is a module such that the ring of endomorphisms, $End(M)$, of $M$ satisfies any of these conditions, then the inequality
\[KS\ell(N)\leq min\left\{\ell(End(M)),GdimEnd(M),dGdimEnd(M)\right\}\] 
holds for all $N\in\mathcal{S}(M)$.

We end this section with a brief study of direct sum decompositions of modules with finite Krull-Schmidt length. We characterize those modules with finite Krull-Schmidt length such that all their indecomposable direct sum decompositions have the same cardinality.

\begin{lemma}
Let $M$ be a module. Let $N$ be a direct summand of $M$. Suppose $M$ has finite Krull-Schmidt length. If all indecomposable direct sum decompositions of $M$ have the same cardinality, then all indecomposable direct sum decompositions of $N$ have the same cardinality.
\end{lemma}

\begin{proof}
Suppose $M$ has finite Krull-Schmidt length. Suppose also that all indecomposable direct sum decompositions of $M$ have the same cardinality. Let $N$ be a direct summand of $M$. Let $N=\bigoplus_{i=1}^nN_i$ and $N=\bigoplus_{j=1}^mN'_j$ be two indecomposable direct sum decompositions of $N$. Let $L\in\mathcal{S}(M)$ be such that $M=N\oplus L$. Let $L=\bigoplus_{k=1}^tL_k$ be an indecomposable direct sum decomposition of $L$. Both direct sum decompositions
\[M=(\bigoplus_{i=1}^nN_i)\oplus(\bigoplus_{k=1}^tL_k) \ \mbox{and}\ M=(\bigoplus_{j=1}^mN'_j)\oplus (\bigoplus_{k=1}^tL_k)\]
of $M$ are indecomposable. The first one of these has cardinality $n+t$, while the second has cardinality $m+t$. It follows that $n+t=m+t$, that is $n=m$. This concludes the proof.  
\end{proof}

\noindent Let $\Omega$ be a class of modules and $\xi:\Omega\rightarrow\mathbb{N}$ be a function. We say that $\xi$ is a rank function if the equation
\[\xi(A\oplus B)\leq \xi(A)+\xi(B)\]
holds for all $A,B\in\Omega$. The functions $pdim$, $Edim$, $Kdim$, and $dKdim$ that associate to each module $M$ its projective, injective, Krull and dual Krull dimension respectively when it is appropriate are all rank functions. Additive functions as defined above are rank functions.
\begin{theorem}
Let $M$ be a module. Suppose $M$ has finite Krull-Schmidt length. The following conditions are equivalent:
\begin{enumerate}
\item Any two indecomposable direct sum decompositions of $M$ have the same cardinality.
\item The function $KS\ell$ is a rank function on $\mathcal{S}(M)$.
\item The function $KS\ell$ is an additive function on $\mathcal{S}(M)$.
\end{enumerate}
Moreover, if $M$ satisfies the conditions above, then for any $N\in\mathcal{S}(M)$, $KS\ell(N)$ is equal to the cardinality of any indecomposable direct decmposition of $M$.
\end{theorem}

\begin{proof}
Observe first that the last statement of the teorem, together with proposition 3.1, imply 3. Suppose now that $KS\ell$ is a rank function on $\mathcal{S}(M)$. We prove, by induction on $k$, that if $N_1,...,N_k\in\mathcal{S}(M)$ are indecomposable and such that $\bigoplus_{i=1}^kN_i\in\mathcal{S}(M)$, then $KS\ell(\bigoplus_{i=1}^kN_i)=k$. The base of the induction is trivial. Suppose now that the claim is true for $k$. Let $N_1,...,N_{k+1}\in\mathcal{S}(M)$ be indecomposable and such that $\bigoplus_{i=1}^{k+1}N_i\in\mathcal{S}(M)$. From the induction hypothesis it follows that $KS\ell(\bigoplus_{i=1}^kN_i)=k$. Hence, from the fact that $KS\ell(N_{k+1})=1$, and from the fact that
\[k<KS\ell(\bigoplus_{i=1}^{k+1}N_i)\leq k+KS\ell(N_{k+1})\]
it follows that $KS\ell(\bigoplus_{i=1}^{k+1}N_i)=k+1$. This proves our claim. From this and proposition 3.1 we conclude that 2 implies the last statement of the theorem. It is easily seen that 3, again with proposition 3.1 also implies the last statement of the theorem. From this, implications $3\Rightarrow 1 $ and $3\Rightarrow 2$ follow. The rest of the implications are trivial. This concludes the proof.
\end{proof}

\section{The $AKS_i$ Conditions}

\noindent We devote this final section to the comparison of certain finiteness conditions under the presence of finiteness conditions discussed in previous sections. Recall [6] that a module $M$ is said to be almost Krull-Schmidt ($AKS$ for short) if it only admits a finite number of direct sum decompositions up to decomposition isomorphisms. We will take the liberty here of saying that $M$ satisfies the $AKS_1$ condition when it is $AKS$. We will, further, say that $M$ satisfies the $AKS_2$ condition if it has only a finite number of indecomposable direct sum decompositions up to decomposition-isomorphisms, that $M$ satisfies the $AKS_3$ condition if $M$ has only a finite number of direct summands up to isomorphisms, and finally, we will say that $M$ satisfies the $AKS_4$ condition if it has only a finite number of indecomposable direct summands up to isomorphisms. It is clear, from our definitions that the $AKS_1$ condition implies the condition of a module having finite Krull-Schmidt length and all other $AKS_i$ conditions. It is also clear that both the $AKS_2$ and $AKS_3$ conditions imply the $AKS_4$ condition. The abelian group $\mathbb{Z}_2^{(\mathbb{N})}$ satisfies the $AKS_2$ and $AKS_4$ conditions, but it does not satisfy the $AKS_1$ condition, and it does not admit deviation on direct summands. [6, ex. 3.3] gives a series of examples of modules not admitting deviation on direct summands, satisfying the $AKS_3$ condition, but not satisfying the $AKS_1$ condition. [6, ex. 4.2] gives examples of noetherian modules, modules with finite Goldie dimension, and modules with finite dual Goldie dimension that do not satisfy the $AKS_1$ condition. Thus the condition of a module having finite Krull-Schmidt length, and thus the condition of a module being Krull-Schmidt artinitan, need not, in general, imply the $AKS_1$ condition. The following proposition says that under the assumption that all direct sum decompositions of a module $M$ refine into finite indecomposable direct sum decompositions, both the $AKS_1$ and $AKS_2$ conditions are equivalent over $M$

\begin{proposition}
Suppose all direct sum decompositions of $M$ refine into finite indecomposable direct sum decompositions (e.g. $M$ is Krull-Schmidt artinian). Then the following conditions are equivalent:
\begin{enumerate}
\item $M$ satisfies the $AKS_1$ condition.
\item $M$ satisfies the $AKS_2$ condition.
\end{enumerate}
\end{proposition}

\begin{proof}
It suffices to prove $2\Rightarrow 1$. Suppose all direct sum decompositions of $M$ refine into finite indecomposable direct sum decompositions. Suppose also that $M$ satisfies the $AKS_2$ condition. Let $M=\bigoplus_{i=1}^m M_i$ be an indecomposable direct sum decomposition of $M$. If $\pi_m$ denotes the number of partitions of the set $\left\{1,...,m\right\}$, then there are exactly $\pi_m$ direct sum decompositions of $M$ which refine into $M=\bigoplus_{i=1}^mM_i$. Thus, if the set $\left\{M=\bigoplus_{j=1}^{m_i}M_{i_j}:1\leq i\leq n\right\}$ is a set of representatives of the set of all indecomposable direct sum decompositions of $M$ up to decomposition-isomorphisms, then there are at most $\prod_{i=1}^n\pi_{m_i}$ direct sum decompositions of $M$ which refine into one of the $M=\bigoplus_{j=1}^{m_i}M_{i_j}$'s. Now, given two decomposition-isomorphic direct sum decompositions of $M$, $M=\bigoplus_{\alpha\in A}M_\alpha$, and $M=\bigoplus_{\beta\in B}M'_\beta$, then clearly every direct sum decomposition of $M$ which refines into $M=\bigoplus_{\alpha\in A}M_\alpha$ is decomposition-isomorphic to a direct sum decomposition of $M$ which refines into $M=\bigoplus_{\beta\in B}M'_\beta$. From this and from the fact that every direct sum decomposition of $M$ refines into a direct sum decomposition of $M$ decomposition-isomorphic to one of the $M=\bigoplus_{j=1}^{m_i}M_{i_j}$'s we conclude that $M$ satifies the $AKS_1$ condition. 
\end{proof}

\noindent The next proposition says that, under the assumption of the condition of a module $M$ being Krull-Schmidt artinian, all three of the $AKS_i$ ($i=1,2,3$) conditions are equivalent over $M$. First we prove the following lemma.

\begin{lemma}
If $M$ is Krull-Schmidt artinian then there do not exist direct summands $N,N'$ of $M$ such that $N'$ is a proper direct summand of $N$, isomorphic to $N$. 
\end{lemma}

\begin{proof}
Suppose there exist direct summands $N$ and $N'$ of $M$ such that $N'$ is a proper direct summand of $N$, isomorphic to $N$. We can  then construct, recursevely, an infinite descending chain $N_1\geq N_2\geq...$ of direct summands of $M$ such that $N_{i+1}$ is a proper direct summand of $N_i$, isomorphic to $N_i$, for all $i\geq 1$, which contradicts the fact that $M$ is Krull-Schmidt artinian.
\end{proof}

\begin{proposition}
Suppose $M$ is Krull-Schmidt artinian. Then the following conditions are equivalent.
\begin{enumerate}
\item $M$ satisfies the $AKS_1$ condition.
\item $M$ satisfies the $AKS_2$ condition.
\item $M$ satisfies the $AKS_3$ condition.
\end{enumerate} 
\end{proposition}

\begin{proof}
By proposition 6.1 it suffices to prove $3\Rightarrow 1$. Suppose that $M$ is Krull-Schmidt artinian.Suppose also that $M$ satisfies the $AKS_3$ condition. By lemma 6.2 there do not exist direct summands $N$ and $N'$ of $M$ such that $N'$ is a proper direct summand of $N$, isomorphic to $N$. Thus, if $M$ has $n$ direct summands up to isomorphisms, then $M$ does not admit direct sum decompositons of cardinality $>n$, since the existence of any such direct sum decomposition, say $M=\bigoplus_{i=1}^mM_i$ of $M$ ($m>n$), would imply that the set $\left\{\bigoplus_{i=1}^kM_i:i\leq m\right\}$ is a set of strictly more than $n$ direct summands of $M$ pairwise non-isomorphic. It clearly follows from this that $M$ does not admit more than $\sum_{i=1}^nn^i$ direct sum decompositions up to decomposition-isomorphisms.
\end{proof}

\noindent Finally, the following proposition says that, under the assumption that a module $M$ has finite Krull-Schmidt length, all $AKS_i$ conditions can be proven to be equivalent over $M$.

\begin{corollary}
Suppose $M$ has finite Krull-Schmidt length. Then the following conditions are equivalent.
\begin{enumerate}
\item $M$ satisfies the $AKS_1$ condition.
\item $M$ satisfies the $AKS_2$ condition.
\item $M$ satisfies the $AKS_3$ condition.
\item $M$ satisfies the $AKS_4$ condition.
\end{enumerate}
\end{corollary}

\begin{proof}
By proposition 6.1 and proposition 6.3 it suffices to prove $4\Rightarrow 1$. Suppose that $M$ has finite Krull-Schmidt length. Suppose also that $M$ satisfies the $AKS_4$ condition. If $KS\ell(M)=n$ and $M$ has $m$ direct summands up to isomorphisms, then $M$ has no more than $\sum_{i=1}^nm^i$ direct sum decompositions up to decomposition-isomorphisms. Thus $M$ satisfies the $AKS_3$ condition. It follows, by proposition 6.3 that $M$ satisfies the $AKS_1$ condition.
\end{proof}

\noindent Observe that from proposition 6.3 and corollary 6.4 a Krull-Schmidt artinian module satisfying the $AKS_4$ condition but not satisfying any of the other $AKS_i$ conditions would be an example of a Krull-Schmidt artinian module with infinite Krull-Schmidt length. Observe also that corollary 6.4 characterizes the complement of the class of all modules satisfying the $AKS_1$ condition inside the class of all modules with finite Krull-Schmidt length as the class of all modules with finite Krull-Schmidt length not satisfying any of the $AKS_i$ conditions. Thus, while every artinian module satisfies all $AKS_i$ conditions [6], there exist, in general, noetherian modules which fail to satisfy all $AKS_i$ conditions.

\section*{Acknowledgements}

\noindent The author would like to thank professor Jos\'e R\'ios-Montes for his interest, for his guidance, and for very helpful revisions of early versions of this paper. The author is also grateful with the anonymous referee for his many valuable corrections and suggestions. Finally, the author would like to dedicate this paper to the late professor Francisco Raggi without whose guidance, encouragement, and friendship, this paper would have not been possible. 

\section{Bibliography}

\noindent [1]	F. Anderson, K. Fuller, \textit{Rings and Categories of Modules}, Springer-Verlag (1991).	

\	

\noindent [2] G. M. Bergman, W. Dicks, \textit{Universal Derivations and Universal Ring Constructions}, Pacific J. Math. \textbf{79} (1978) 293-337.

\

\noindent [3] R. Camps, W. Dicks, \textit{On Semilocal Rings}, Israel J. Math. \textbf{81} (1993), 203-211.

\

\noindent [4]	D. S. Dummit, R. M. Foote, \textit{Abstract Algebra}, John Wiley, and Sons (2004).

\

\noindent [5]	A. Facchini, \textit{Module Theory, Endomorphism Rings and Direct Sum Decompositions in Some Classes of Modules}, Birkhauser-Verlag (1998).

\	

\noindent [6]	A. Facchini, D. Herbera, \textit{Modules With Only Finitely Direct Sum Decompositions up to Isomorphism}, Irish Mathematical Society Bulletin \textbf{50} (2003), 51-69.
	
\

\noindent [7]  B. L. Osofsky, \textit{A Remark on the Krull-Schmidt-Azumaya Theorem}, Canad. Math. Bull. \textbf{13} (1970) 501–505. 

\

\noindent [8] P. Vamos, \textit{The Holy Grail of Algebra: Seeking Complete Sets of Invariants}, Abelian Groups and Modules, A. Facchini, C. Menini, eds., Math and its Appl. 343, Kluwer, Dordrecht, 1995, 475-483.

\end{document}